\documentclass[letterpaper, 10 pt, conference]{ieeeconf}  

\IEEEoverridecommandlockouts
\usepackage{graphicx}
\usepackage[labelformat=simple]{subcaption}

\usepackage{xcolor}
\usepackage{amsmath} 
\usepackage{mathrsfs}
\usepackage{amssymb}
\usepackage{mathtools}
\usepackage{mathdots}
\usepackage{cite}
\usepackage[hidelinks]{hyperref}
\let\labelindent\relax
\usepackage[shortlabels]{enumitem}
\usepackage{algorithm}
\usepackage[noend]{algpseudocode}
\usepackage{bm}
\usepackage{url}
\usepackage{optidef}
\usepackage{multirow, multicol}
\usepackage{balance}
\newtheorem{theorem}{\bf Theorem}[section]

\newtheorem{lemma}{\bf Lemma}[section]

\newtheorem{definition}{\bf Definition}[section]
\newtheorem{remark}{\bf Remark}[section]

\newcommand{\longthmtitle}[1]{\mbox{}\textup{\textbf{(#1):}}}

\newcommand{\real}{{\mathbb{R}}}

\newcommand{\realpositive}{{\mathbb{R}}_{>0}}
\newcommand{\realnonnegative}{{\mathbb{R}}_{\ge 0}}

\newcommand{\integersnonnegative}{\mathbb{Z}_{\geq 0}}
\newcommand{\integerspositive}{\mathbb{Z}_{> 0}}

\newcommand{\eps}{\epsilon}

\newcommand{\oprocendsymbol}{\hbox{$\bullet$}}
\newcommand{\oprocend}{\relax\ifmmode\else\unskip\hfill\fi\oprocendsymbol}

\newcommand{\defeq}{\vcentcolon=}

\DeclareMathOperator*{\argmin}{arg\,min}

\newcommand{\Tini}{{T_{\textup{ini}}}}
\newcommand{\Tf}{{T_{\textup{f}}}}

\newcommand{\yini}{{y_{\textup{ini}}}}

\newcommand{\Uf}{{U_{\mathrm{f}}}}
\newcommand{\Yp}{{Y_{\mathrm{p}}}}
\newcommand{\Yf}{{Y_{\mathrm{f}}}}
\newcommand{\Uphat}{{\widehat{U}_{\mathrm{p}}}}
\newcommand{\Ufhat}{{\widehat{U}_{\mathrm{f}}}}
\newcommand{\Yphat}{{\widehat{Y}_{\mathrm{p}}}}
\newcommand{\Yfhat}{{\widehat{Y}_{\mathrm{f}}}}
\newcommand{\col}{{\textup{col}}}
\newcommand{\Prob}{\mathbb{P}}
\newcommand{\Exp}{\mathbb{E}}
\newcommand\scalemath[2]{\scalebox{#1}{\mbox{\ensuremath{\displaystyle #2}}}}

\title{Regularized and Distributionally Robust Data-Enabled Predictive Control}

\author{Jeremy Coulson \qquad John Lygeros \qquad Florian D\"{o}rfler
\thanks{All authors are with the Department of Information Technology and Electrical Engineering at ETH Z\"{u}rich, Switzerland~\texttt{\{jcoulson, lygeros, dorfler\}@control.ee.ethz.ch}.}}

\begin{document}

\maketitle
\thispagestyle{empty}
\pagestyle{empty}

\begin{abstract}

In this paper, we study a data-enabled predictive control (DeePC) algorithm applied to unknown stochastic linear time-invariant systems. The algorithm uses noise-corrupted input/output data to predict future trajectories and compute optimal control policies. To robustify against uncertainties in the input/output data, the control policies are computed to minimize a worst-case expectation of a given objective function. Using techniques from distributionally robust stochastic optimization, we prove that for certain objective functions, the worst-case optimization problem coincides with a regularized version of the DeePC algorithm. These results support the previously observed advantages of the regularized algorithm and provide probabilistic guarantees for its performance. We illustrate the robustness of the regularized algorithm through a numerical case study.

\end{abstract}

\section{Introduction}

Data-driven control approaches have become increasingly popular in recent years due to the growing complexity of systems~\cite{ZSH-ZW:13}. The increase in system complexity poses significant challenges to model-based control design, in which acquiring an accurate model for the system is the most crucial step, and often the most time-consuming and expensive step~\cite{LL:86-book,HH:05}. The goal of data-driven control approaches is to bypass the system identification step, and directly use system measurements for control design. 

The focus of this paper is on the problem of data-driven control for unknown systems. We seek to develop a data-driven control algorithm that uses input/output data from the unknown system to compute control inputs which satisfy system constraints. Furthermore, since uncertainties and noise corruption in the data measurements are inevitable in any real-world application, we seek algorithms that are robust to the uncertain and noisy data measurements used. 

Several approaches such as safe-learning, learning based Model Predictive Control (MPC), and stochastic MPC aim at solving similar problems (see, e.g.,~\cite{KPW-MNZ:18, FB-MT-AS-AK:17, AA-HG-SSS-CT:13,AM:16} and references therein). However, these methods rely on having an a-priori known model for the system, require full-state measurement, and often assume knowledge about the disturbances that act on the system (e.g., bounded disturbances within a known set). Another large area of connected literature is reinforcement learning (see~\cite{FL-DV-KGV:12,BR:18} and references therein). These methods also address similar problems, but usually require a large number of data measurements. Additionally, these approaches sometimes result in unreliable outcomes due to the sensitivity of the algorithms to hyper-parameters~\cite{RI-HP-MG-DP:17}. Hence, none of the approaches above are suitable for real-time, constrained, and optimal control based only on input/output samples.

We employ an alternative approach which does not rely on a particular parametric system model or assume any knowledge of the disturbances. The approach uses behavioural system theory to characterize the behaviour of the unknown system (i.e., the possible trajectories of the system)~\cite{JCW-PR-IM-BDM:05}. This behavioural approach was used for control design in~\cite{IM-PR:08} where optimal open-loop control policies were computed using input/output data from the unknown system. This was later extended in~\cite{JC-JL-FD:18} to a receding horizon set up, which was proved to have equivalent closed-loop performance when compared to standard MPC in the case of deterministic linear time invariant (LTI) systems. 

The approach presented in this paper is built upon the \textbf{D}ata-\textbf{e}nabl\textbf{e}d \textbf{P}redictive \textbf{C}ontrol (DeePC) algorithm for deterministic linear systems presented in~\cite{JC-JL-FD:18}. The DeePC algorithm uses raw input/output data to compute optimal controls for unknown systems using real-time output feedback via a receding horizon implementation, thus allowing for the consideration of input/output constraints. In turn, this approach is much simpler to implement than the safe-learning, MPC, and reinforcement learning approaches above, which in the case of an unknown system with only output measurements available, require system identification and state observer design. On the other hand, the DeePC algorithm may not be amenable to stochastic systems, as the theory is built on the assumption of a deterministic LTI system. However, it was observed in~\cite{JC-JL-FD:18} that after adding heuristic regularizations to the algorithm, DeePC still performed well on systems in the presence of stochastic disturbances, yet no robustness guarantees were given. We focus on extending this algorithm for use on stochastic systems. In doing so, we will give rigorous meaning to the heuristic regularizations presented in~\cite{JC-JL-FD:18} by means of probabilistic guarantees on robust performance.

Due to the receding horizon implementation of the DeePC algorithm, we are required to repeatedly solve an optimization problem whose objective function naturally depends on data measurements that are affected by unknown disturbances. In order to be robust to these disturbances, we use distributionally robust optimization techniques~\cite{PME-DK:18, SSA-DK-PME:17} which give rise to a regularized version of the DeePC algorithm similar to the heuristics presented in~\cite{JC-JL-FD:18}. 

\emph{Contributions:}  Motivated by distributionally robust optimization and its connection to regularized optimization~\cite{PME-DK:18, SSA-DK-PME:17}, we develop a novel robust DeePC algorithm which uses input/output data to compute optimal and robust control inputs for unknown linear stochastic systems. We give probabilistic guarantees on its performance, giving rigorous justification for the use of the regularized control algorithm given in~\cite{JC-JL-FD:18}. As a by-product, we gain new insights into the design of such regularizers.

\emph{Organization:} Section~\ref{sec:problem} contains the problem statement. In Section~\ref{sec:preliminaries}, we recall the DeePC algorithm and show how robust stochastic optimization techniques can be applied to improve the algorithm. Section~\ref{sec:mainresults} contains the main results showing that regularizations on the DeePC algorithm result in robust performance. We illustrate the performance of the regularized algorithm in Section~\ref{sec:simulations}. The paper is concluded in Section~\ref{sec:conclusion}.

\emph{Notation:} Given $x,y\in\real^n$, $\langle x,y\rangle\defeq x^Ty$ denotes the usual inner product on $\real^n\times\real^n$. We denote the dual norm of a norm $\|\cdot\|$ on $\real^n$ by $\|x\|_{\ast}\defeq\sup_{\|y\|\leq 1}\langle x,y\rangle$. The conjugate function of a function $f:\real^n\to\real$ is denoted by $f^{\ast}(\theta)\defeq \sup_{x\in\real^n} \langle \theta,x \rangle-f(x)$. We denote the stacked column vector of $x$ and $y$ by $\col(x,y)=(x^T,y^T)^T$. We denote by $\delta_x$ the Dirac distribution at $x$. We use the notation $\widehat{\cdot}$ to denote objects that depend on measured data.

\section{Problem Statement}\label{sec:problem}
Consider the stochastic LTI system
\begin{equation}\label{eq:stochasticsystem}
\begin{cases}
x(t+1)=Ax(t)+Bu(t)+Ev(t) \\
y(t)=Cx(t)+Du(t)+Fv(t),
\end{cases}
\end{equation}
where $A\in\real^{n\times n}$, $B\in\real^{n\times m}$, $C\in\real^{p\times n}$, $D\in\real^{p\times m}$, $E\in \real^{n\times q}$, $F\in \real^{p\times q}$, and $x(t)\in \real^n$, $u(t)\in \real^m$, $y(t)\in \real^p$, $v(t)\in \real^q$ are respectively the state, control input, output, and stochastic disturbance of the system at time $t\in \integersnonnegative$. The disturbance $v(t)$ is drawn from an unknown probability distribution $\Prob_v$ for all $t\in\integersnonnegative$. We assume throughout the paper that the pair $(A,B)$ is controllable, the pair $(A,C)$ is observable, that system~\eqref{eq:stochasticsystem} itself is unknown (i.e., $A$, $B$, $C$, $D$, $E$, $F$ unknown), and that we have access only to input/output measurements, which we will denote by $\hat{u}(t),\hat{y}(t)$. 

We consider a finite-horizon optimal control problem, in which the task is to design control inputs for the unknown system to minimize a given objective function. More specifically, given a time horizon $\Tf\in\integerspositive$, an objective function $f\colon \real^{m\Tf}\times\real^{p\Tf}\to \real$, and the current time $t\in\integersnonnegative$, we wish to choose a sequence of control inputs $u=\col(u(t),\dots,u(t+\Tf-1))\in\mathcal{U}$, where $\mathcal{U}\subset\real^{m\Tf}$ is an input constraint set, such that the resulting stochastic output trajectory, $y=\col(y(t),\dots,y(t+\Tf-1))$, of~\eqref{eq:stochasticsystem} minimizes the expectation of the objective function, i.e., we wish to solve the following optimization problem:
\begin{equation}\label{eq:optexpectation}
\underset{u\in\mathcal{U}}{\inf}\quad
\Exp_{\Prob_v^{\Tf}}\left\{f(u,y)\right\},
\end{equation}
where $\Prob_v^{\Tf}=\Prob_v\times\cdots\times\Prob_v$ is the $\Tf$-fold product distribution.
Such optimization problems also appear in stochastic MPC~\cite{PH-EC-DC-FR-JL:12}.

Solving problem~\eqref{eq:optexpectation} poses two main challenges: the system~\eqref{eq:stochasticsystem} is unknown, and the disturbance distribution $\Prob_v$ is unknown. Hence, we are unable to predict future output trajectories $y$, and we are missing necessary information to compute the expectation. Note that even in the case when system~\eqref{eq:stochasticsystem} and $\Prob_v$ are known, solving~\eqref{eq:optexpectation} would require high-dimensional integration and is often computationally intractable~\cite{GAH-DK-WW:16}. One could also consider including joint output chance constraints in the problem setup~\eqref{eq:optexpectation} which would pose similar challenges as above. We do not consider such constraints as they are beyond the scope of this paper.

To address these challenges, we first simplify the problem by considering the deterministic version of system~\eqref{eq:stochasticsystem}:
\begin{equation}\label{eq:deterministicsystem}
\begin{cases}
x(t+1)=Ax(t)+Bu(t) \\
y(t)=Cx(t)+Du(t).
\end{cases}
\end{equation}
In this case, problem~\eqref{eq:optexpectation} reduces to
\begin{equation}\label{eq:optfunction}
\inf_{u\in\mathcal{U}}f(u,y).
\end{equation}
We recall a data-enabled predictive control (DeePC) algorithm first presented in~\cite{JC-JL-FD:18} that can be used to solve~\eqref{eq:optfunction} for deterministic though unknown systems~\eqref{eq:deterministicsystem} (see Section~\ref{sec:DeePCalg}). We then show how this algorithm can be extended to approach problem~\eqref{eq:optexpectation} for stochastic systems~\eqref{eq:stochasticsystem} using distributionally robust optimization techniques (see Section~\ref{sec:robustoptimization}). The main results can be found in Section~\ref{sec:mainresults}.

\section{Preliminaries}\label{sec:preliminaries}

\subsection{DeePC Algorithm for Deterministic Systems}\label{sec:DeePCalg}
In this section, we recall the DeePC algorithm which uses raw input/output data to construct a non-parametric predictive model that was first developed using a behavioural system theory approach in~\cite{JCW-PR-IM-BDM:05}. Let $L,q,T\in\integerspositive$. We define the Hankel matrix of a signal $w\in\real^{qT}$ with $L$ block rows as the matrix
\[
\mathscr{H}_L(w)\defeq
\begin{pmatrix}
w(1) &w(2) &\dots &w(T-L+1)\\
w(2) &w(3) &\dots &w(T-L+2)\\
\vdots &\vdots &\ddots &\vdots \\
w(L) &w(L+1) &\dots &w(T)
\end{pmatrix}.
\]
We say that signal $w$ is \emph{persistently exciting of order $L$} if $\mathscr{H}_L(w)$ has full row rank. Note that in order for $w$ to be persistently exciting of order $L$, we must have $T\geq(q+1)L-1$, i.e., the signal $w$ has to be sufficiently rich and sufficiently long. The \emph{lag} of system~\eqref{eq:deterministicsystem} is defined as the smallest integer $\ell\in \integerspositive$ such that the observability matrix $\mathscr{O}_{\ell}(A,C)\defeq\col\left(C,CA, \dots,CA^{\ell-1}\right)$ has rank $n$ (see~\cite[Section 7.2]{IM-JCW-SVH-BDM:06-book} for equivalent state space free definitions of lag). 

Like in Section~\ref{sec:problem}, we assume that we have access to input/output measurements. Let $\col(\hat{u},\hat{y})=\col(\hat{u}(1),\dots,\hat{u}(T),\hat{y}(1),\dots,\hat{y}(T))\in\real^{(m+p)T}$ be a measured trajectory of~\eqref{eq:deterministicsystem} of length $T\in\integerspositive$. Assume that $\hat{u}$ is persistently exciting of order $\Tini+\Tf+n$, where $\Tini,\Tf\in\integerspositive$. We organize the data into the Hankel matrices
\begin{equation}\label{eq:UpUfYpYf}
\begin{pmatrix}
\Uphat \\ \Ufhat
\end{pmatrix}\defeq \mathscr{H}_{\Tini+\Tf}(\hat{u}), \quad
\begin{pmatrix}
\Yphat \\ \Yfhat 
\end{pmatrix}\defeq \mathscr{H}_{\Tini+\Tf}(\hat{y}),
\end{equation}
where $\Uphat$ consists of the first $\Tini$ block rows of $\mathscr{H}_{\Tini+\Tf}(\hat{u})$ and $\Ufhat$ consists of the last $\Tf$ block rows of $\mathscr{H}_{\Tini+\Tf}(\hat{u})$ (similarly for $\Yphat$ and $\Yfhat$). Given the current time $t\in\integersnonnegative$, let $\hat{u}_{\textup{ini}}=\col(\hat{u}(t-\Tini),\dots,\hat{u}(t-1))\in\real^{m\Tini}$ and $\hat{y}_{\textup{ini}}=\col(\hat{y}(t-\Tini),\dots,\hat{y}(t-1))\in\real^{p\Tini}$ be the $\Tini$ most recent input and output measurements of~\eqref{eq:deterministicsystem}, respectively. By~\cite[Theorem 1]{JCW-PR-IM-BDM:05}, $\col(\hat{u}_{\textup{ini}},u,\hat{y}_{\textup{ini}},y)\in\real^{(m+p)(\Tini+\Tf)}$ is a trajectory of system~\eqref{eq:deterministicsystem} if and only if there exists $g\in\real^{T-\Tini-\Tf+1}$ such that
\begin{equation}\label{eq:hankel}
\begin{pmatrix}
\Uphat \\
\Yphat \\
\Ufhat \\
\Yfhat
\end{pmatrix}g=
\begin{pmatrix}
\hat{u}_{\textup{ini}} \\
\hat{y}_{\textup{ini}} \\
u \\
y
\end{pmatrix}.
\end{equation}
Furthermore, if $\Tini \geq \ell$, then by~\cite[Lemma 1]{IM-PR:08}, there exists a unique $x(t)\in \real^n$ such that the output trajectory $y$ is uniquely determined by the system~\eqref{eq:deterministicsystem}. In other words, the trajectory $\col(\hat{u}_{\textup{ini}},\hat{y}_{\textup{ini}})$ fixes the underlying initial state $x(t)$ from which the trajectory $\col(u,y)$ evolves. Hence, the Hankel matrix in~\eqref{eq:hankel} serves as a non-parametric predictive model for system~\eqref{eq:deterministicsystem}. This observation was first exploited for control purposes in~\cite{IM-PR:08}, where equation~\eqref{eq:hankel} was used to construct open-loop control policies for tracking. This was then extended in~\cite{JC-JL-FD:18}, where equation~\eqref{eq:hankel} was used in a receding horizon optimal control algorithm. We recall the latter approach below.

Given a time horizon $\Tf \in \integerspositive$, past input/output data $\col(\hat{u}_{\textup{ini}},\hat{y}_{\textup{ini}})\in \real^{(m+p)\Tini}$, objective function $f\colon \real^{m\Tf}\times\real^{p\Tf}\to \real$, we formulate the following optimization problem:
\begin{align}\label{eq:DeePC}
\underset{g}{\text{minimize}}\quad
& f(\Ufhat g,\Yfhat g)\nonumber \\
\text{subject to\quad}
& \begin{pmatrix}
\Uphat \\ \Yphat
\end{pmatrix}g
=\begin{pmatrix}
\hat{u}_{\textup{ini}} \\ 
\hat{y}_{\textup{ini}}
\end{pmatrix}\\
&\Ufhat g \in\mathcal{U}\nonumber.
\end{align}
Note that optimization problem~\eqref{eq:DeePC} is equivalent to~\eqref{eq:optfunction} since by~\eqref{eq:hankel}, $u=\Ufhat g$ and $y=\Yfhat g$. The data-enabled predictive control (DeePC) algorithm is then given as follows:
\begin{algorithm}
\caption{Deterministic DeePC}\label{alg:DeePC}
\textbf{Input:} trajectory $\col(\hat{u},\hat{y})\in\real^{(m+p)T}$ with $\hat{u}$ persistently exciting of order $\Tini+\Tf+n$, most recent input/output measurements $\col(\hat{u}_{\textup{ini}},\hat{y}_{\textup{ini}})\in\real^{(m+p)\Tini}$
\begin{enumerate}
\item Solve~\eqref{eq:DeePC} for $g^{\star}$.\label{step:DeePCbegin}
\item Compute the optimal input sequence $u^{\star}=\Ufhat g^{\star}$.
\item Apply input $(u(t),\dots,u(t+s))=(u_0^{\star},\dots,u_s^{\star})$ for some $s\leq \Tf-1$.
\item Set $t$ to $t+s$ and update $\hat{u}_{\textup{ini}}$ and $\hat{y}_{\textup{ini}}$ to the $\Tini$ most recent input/output measurements.\label{step:DeePCend}
\item Return to~\ref{step:DeePCbegin}.
\end{enumerate}
\end{algorithm}

This algorithm was shown to be equivalent to the classical Model Predictive Control (MPC) algorithm when considering deterministic LTI systems of the form~\eqref{eq:deterministicsystem}, also in the presence of additional output constraints. 
When Algorithm~\ref{alg:DeePC} was applied to a nonlinear stochastic quadcopter model in~\cite{JC-JL-FD:18}, additional heuristic regularization terms had to be included in~\eqref{eq:DeePC} to achieve robust performance. 
We recall the regularizations below.
\begin{align}\label{eq:heuristicDeePC}
\underset{g}{\text{minimize}}\quad
& f(\Ufhat g,\Yfhat g)+\lambda_{\textup{ini}}\|\Yphat g-\hat{y}_{\textup{ini}}\|_1+\lambda_g\|g\|_1\nonumber \\
\text{subject to\quad}
& \Uphat g = \hat{u}_{\textup{ini}}\\
&\Ufhat g\in\mathcal{U}\nonumber,
\end{align}
where $\lambda_{\textup{ini}},\lambda_g\in\realnonnegative$.
We will show that this regularized optimization problem coincides with solving a distributionally robust variation of optimization problem~\eqref{eq:optexpectation}, giving rigorous justification  for the regularizations (see Section~\ref{sec:mainresults}).

\subsection{Distributionally Robust DeePC}\label{sec:robustoptimization}
Consider now the stochastic system~\eqref{eq:stochasticsystem}. Let $\col(\hat{u},\hat{y})\in\real^{(m+p)T}$ be a measured trajectory of length $T$ of system~\eqref{eq:stochasticsystem} such that $\hat{u}$ is persistently exciting of order $\Tini+\Tf+n$, where $\Tini\geq \ell$. Furthermore, assume $\col(\hat{u}_{\textup{ini}},\hat{y}_{\textup{ini}})\in\real^{(m+p)\Tini}$ is the last measured trajectory of system~\eqref{eq:stochasticsystem}. The output trajectories $\hat{y}$ and $\hat{y}_{\textup{ini}}$ can be viewed as particular realizations of random variables, $y$ and $\yini$. If we naively use the noise corrupted trajectory $\hat{y}$ to build Hankel matrices $\Uphat$, $\Yphat$, $\Ufhat$, $\Yfhat$ as in~\eqref{eq:UpUfYpYf} and apply Algorithm~\ref{alg:DeePC}, we will run into difficulties regarding the consistency of the constraint equations in~\eqref{eq:DeePC}. Indeed, since $\hat{y}_{\textup{ini}}$ and $\Yphat$ consist of noise corrupted outputs, there may not exist $g$ that satisfies the equation $\Yphat g=\hat{y}_{\textup{ini}}$. Hence, we soften the equality constraint and penalize the slack variable with an appropriate cost function; we recall that the use of such ``soft constraints'' is common in MPC~\cite{MZ-MM-CJ:14}. This results in the optimization problem
\begin{align*}
\underset{g}{\text{minimize}}\quad
& f(\Ufhat g,\Yfhat g)+\lambda_{\textup{ini}}\|\Yphat g-\hat{y}_{\textup{ini}}\|_1 \\
\text{subject to\quad}
& \Uphat g=\hat{u}_{\textup{ini}}\\
&\Ufhat g\in\mathcal{U},
\end{align*}
where we have used the 1-norm as a penalty function on the slack variable. It is well known that if $\lambda_{\textup{ini}}$ is chosen large enough then the solution $g^{\star}$ will violate $\Yphat g^{\star}=\hat{y}_{\textup{ini}}$ only if  the equation is infeasible~\cite{EK-JM:00}. Hence, the original constraint will be satisfied if it can be satisfied. Since the constraint is now deterministic, we may define $G = \{g\in\real^{T-\Tini-\Tf+1}\mid \Uphat g=\hat{u}_{\textup{ini}},\;\Ufhat g\in\mathcal{U}\}$ and rewrite the above as
\begin{equation}\label{eq:optsoft}
\underset{g\in G}{\text{minimize}}\quad
f(\Ufhat g,\Yfhat g)+\lambda_{\textup{ini}}\|\Yphat g-\hat{y}_{\textup{ini}}\|_1.
\end{equation}
For ease of notation, we put all random objects into a matrix whose $j$-th row we denote by $\xi_j$, i.e.,
\[
\xi_j \defeq \begin{pmatrix}
\Yp &\yini \\
\Yf &0
\end{pmatrix}_{j,\cdot},
\]
where $\begin{pmatrix}\cdot\end{pmatrix}_{j,\cdot}$ denotes the $j$-th row of a matrix. For all $j\in\{1,\dots,p(\Tini+\Tf)\}$, let $\Prob_j$ denote the probability distribution of $\xi_j$ supported on $\Xi_j\subseteq\real^{T-\Tini-\Tf+2}$. Note that the distributions $\Prob_j$ and their support sets $\Xi_j$ are determined by the dynamics of system~\eqref{eq:stochasticsystem} and the unknown distribution $\Prob_v$. Define $\xi = \col(\xi_1^T,\dots,\xi_{p(\Tini+\Tf)}^T)$ and let $\Prob\defeq\Prob_1\times\cdots\times \Prob_{p(\Tini+\Tf)}$ denote the unknown probability distribution of $\xi$ supported on $\Xi=\prod_{j=1}^{p(\Tini+\Tf)}\Xi_j\subseteq \real^{p(\Tini+\Tf)(T-\Tini-\Tf+2)}$. 
Let $H=\{h=\col(g,-1)\mid \Uphat g=\hat{u}_{\textup{ini}},\;\Ufhat g\in\mathcal{U}\}$. Substituting this notation into~\eqref{eq:optsoft} yields
\begin{align*}
\underset{h\in H}{\text{minimize}}\quad
&f((\Ufhat,0) h,(\hat{\xi}_{p\Tini+1}h,\dots,\hat{\xi}_{p(\Tini+\Tf)}h))\nonumber \\
&+\lambda_{\textup{ini}}\|(\hat{\xi}_1h,\dots,\hat{\xi}_{p\Tini}h)\|_1\},
\end{align*}
where $\hat{\xi}=\col(\hat{\xi}_1,\dots,\hat{\xi}_{p(\Tini+\Tf)})$ are our measurements. Denoting the empirical distribution of our measurements $\hat{\xi}$ by $\widehat{\Prob}=\delta_{\hat{\xi}}$, and substituting this notation into the above yields
\begin{align}\label{eq:insample}
\underset{h\in H}{\text{minimize}}\quad
&\Exp_{\widehat{\Prob}}\{c(\xi,h)\},
\end{align}
where $c(\xi,h)=f((\Ufhat,0) h,(\xi_{p\Tini+1}h,\dots,\xi_{p(\Tini+\Tf)}h))+\lambda_{\textup{ini}}\|(\xi_1h,\dots,\xi_{p\Tini}h)\|_1$.

The quantity in~\eqref{eq:insample} is known as the \emph{in-sample performance}. Unfortunately, the solution $h^{\star}$ to~\eqref{eq:insample} may result in poor \emph{out-of-sample performance} 
\begin{equation}\label{eq:outofsample}
\Exp_{\Prob}\{c(\xi,h^{\star})\},
\end{equation}
which if we compare to equation~\eqref{eq:optexpectation} is the real quantity of interest (see, e.g.,~\cite{PME-DK:18,DB-VG-NK:17} for examples displaying poor out-of-sample performance). 

To alleviate this problem, we focus instead on a robust variation of~\eqref{eq:insample} which will serve as an upper bound for out-of-sample performance~\eqref{eq:outofsample} with high confidence.
In particular, we seek solutions of
\begin{equation}\label{eq:robustproblem}
\inf_{h\in H}\sup_{Q\in\widehat{\mathcal{P}}}\Exp_Q\{c(\xi,h)\},
\end{equation}
where $\widehat{\mathcal{P}}$ is an \emph{ambiguity set} which depends on the sampled data trajectories $\hat{\xi}$. The ambiguity set will be constructed in such a way that it contains the actual distribution $\Prob$ with high confidence. Thus, if $h^{\star}$ and $J^{\star}$ are the solution and optimal value of~\eqref{eq:robustproblem}, then $J^{\star}$ will upper bound the out-of-sample performance with high confidence. 

Following~\cite{PME-DK:18}, we define the ambiguity set in~\eqref{eq:robustproblem} in terms of the Wasserstein metric defined on the space $\mathcal{M}(\Xi)$ denoting the set of all distributions $Q$ supported on $\Xi$ such that $\Exp_Q[\|\xi\|_{\textup{W}}]<\infty$, where $\|\cdot\|_{\textup{W}}$ is an arbitrary norm.
\begin{definition}
The \emph{Wasserstein metric} $\textup{d}_\textup{W}\colon\mathcal{M}(\Xi)\times\mathcal{M}(\Xi)\to\realnonnegative$ is defined as
\[
\textup{d}_\textup{W}(Q_1,Q_2)\defeq\inf_{\Pi}\left\{\int_{\Xi^2}\|\xi_1-\xi_2\|_{\textup{W}}\Pi(\textup{d}\xi_1,\textup{d}\xi_2)\right\},
\]
where $\Pi$ is a joint distribution of $\xi_1$ and $\xi_2$ with marginal distributions $Q_1\in \mathcal{M}(\Xi)$ and $Q_2\in\mathcal{M}(\Xi)$ respectively.
\end{definition}
\smallskip

Let $\eps\geq0$. We denote the \emph{Wasserstein ball of radius $\eps$ centred around distribution $Q$} by $B_{\eps}(Q)\defeq\{Q'\in\mathcal{M}(\Xi)\mid \textup{d}_{\textup{W}}(Q,Q')\leq \eps\}$. The Wasserstein metric can be viewed as a distance between probability distributions, where the distance is calculated via an optimal mass transport plan $\Pi$. Note that there are other ways to construct ambiguity sets (see, e.g.,~\cite{DB-VG-NK:17} where the ambiguity set is constructed as the confidence region of a goodness-of-fit hypothesis test). Replacing the general ambiguity set $\widehat{\mathcal{P}}$ with a Wasserstein ball around the empirical distribution $\widehat{\Prob}$ results in the problem
\begin{equation}\label{eq:optwasserstein}
\inf_{h\in H}\sup_{Q\in B_{\eps}(\widehat{\Prob})}\Exp_Q\{c(\xi,h)\}.
\end{equation}
In the next section, we show that if $h^{\star}$ and $J^{\star}$ are the solution and optimal value of~\eqref{eq:optwasserstein}, then $J^{\star}$ will upper bound the out-of-sample performance~\eqref{eq:outofsample} with high confidence. Hence, we obtain probabilistic guarantees that applying control inputs $u^{\star}=\Ufhat h^{\star}$ to system~\eqref{eq:stochasticsystem} will result in good performance of the resulting stochastic trajectory $y$.
Additionally, we show that~\eqref{eq:optwasserstein} is computationally tractable and results in a regularized version of the DeePC algorithm similar to~\eqref{eq:heuristicDeePC}.

\section{Main Results}\label{sec:mainresults}
The following result relates the robust problem~\eqref{eq:optwasserstein} to the out-of-sample performance~\eqref{eq:outofsample}. In particular, if $h^{\star}$ and $J^{\star}$ are the solution and the optimal value of the robust problem~\eqref{eq:optwasserstein}, then $J^{\star}$ upper bounds the out-of-sample performance~\eqref{eq:outofsample} with high confidence.
\begin{theorem}\label{thm:performanceguarantee}\longthmtitle{Robust Performance Guarantee}
Assume that distribution $\Prob$ is light-tailed, i.e., there exists $a>1$ such that $\Exp_{\Prob}[e^{\|\xi\|_{\textup{W}}^a}]<\infty$. Let $\beta\in(0,1)$. Then there exists $\eps=\eps(\beta)>0$ such that for all $h\in H$,
\[
\Prob\left\{\Exp_{\Prob}\{c(\xi,h)\}\leq\sup_{Q\in B_{\eps}(\widehat{\Prob})}\Exp_Q\{c(\xi,h)\}\right\}\geq 1-\beta.
\]
\end{theorem}
\smallskip
The proof of the above theorem follows directly from~\cite[Theorem 3.5]{PME-DK:18}. Note that the light-tailed assumption is satisfied automatically when $\Xi$ is compact. Hence, all distributions truncated to a compact support set satisfy the assumptions of Theorem~\ref{thm:performanceguarantee}. Other examples include Gaussian and exponential distributions.

By adapting the proof methods of~\cite{PME-DK:18} and~\cite{SSA-DK-PME:17} to our setting, we show that for certain objective functions, the semi-infinite optimization problem~\eqref{eq:optwasserstein} reduces to a tractable convex program that coincides with a regularized version of~\eqref{eq:insample}. Let $\|\cdot\|$ be an arbitrary norm on $\real^{T-\Tini-\Tf+2}$ and $\|\xi\|_{\textup{W}}=\sum_{j=1}^{p(\Tini+\Tf)}\|\xi_j\|$ be the norm used in the Wasserstein metric.
\begin{theorem}\longthmtitle{Tractable Reformulation}\label{thm:reformulation}
Assume that the objective function $f$ is separable and can be written as $f(u,y)=f_1(u)+f_2(y)$ for all $u,y$, where $f_1$ and $f_2$ are convex and continuous. Furthermore, assume $f_2$ is such that $\Theta_2=\{\theta\mid f_2^{\ast}(\theta)<\infty\}$ is a bounded set in $\real^{p\Tf}$. Then
\begin{align*}
&\underset{Q\in B_{\eps}(\widehat{\Prob})}{\text{sup}}\Exp_{Q}\left\{c(\xi,h)\right\}\nonumber\\
&\leq c(\hat{\xi},h)+ \eps\cdot \max\left\{\sup_{\theta\in\Theta_2}\|\theta\|_{\infty}\|\col(g,0)\|_{\ast},\lambda_{\textup{ini}}\|h\|_{\ast}\right\},
\end{align*}
where $\|\cdot\|_{\ast}$ denotes the dual norm of $\|\cdot\|$. Equality holds when $\Xi_j=\real^{T-\Tini-\Tf+2}$ for all $j\in\{1,\dots,p\Tini\}$ and $\Xi_j=\real^{T-\Tini-\Tf+1}\times\{0\}$ for all $j\in\{p\Tini+1,\dots,p(\Tini+\Tf)\}$.
\end{theorem}
\smallskip

Note that the condition for equality holds when the disturbance $v$ affecting system~\eqref{eq:stochasticsystem} is drawn from a distribution $\Prob_v$ defined on an unbounded support set and the matrix $F$ in~\eqref{eq:stochasticsystem} is full row rank. In other words, equality will hold if the probability distribution $\Prob_v$ has unbounded support and disturbances $v(t)$ affect all entries of the output vector $y(t)$. In this case, each probability distribution $\Prob_j$ would have support $\Xi_j=\real^{T-\Tini-\Tf+2}$ for $j\in\{1,\dots,p\Tini\}$ and $\Xi_j=\real^{T-\Tini-\Tf+1}\times\{0\}$ for $j\in\{p\Tini+1,\dots,p(\Tini+\Tf)\}$. For example, we would obtain equality in the above if $F$ is full row rank and $\Prob_v$ is a Gaussian distribution which is a common assumption in stochastic MPC (see~\cite{AM:16} and references therein). We also see an immediate connection between Theorem~\ref{thm:reformulation} and~\eqref{eq:heuristicDeePC}. In fact, when $\|\cdot\|_{\ast}=\|\cdot\|_1$ then $\eps\cdot \max\{\sup_{\theta\in\Theta_2}\|\theta\|_{\infty}\|\col(g,0)\|_{\ast},\lambda_{\textup{ini}}\|h\|_{\ast}\}=\eps\cdot \max\{\sup_{\theta\in\Theta_2}\|\theta\|_{\infty}\|g\|_{1},\lambda_{\textup{ini}}(\|g\|_{1}+1)\}$. Hence, depending on the known cost function $f_2$ (i.e., on the set $\Theta_2$), $\lambda_g$ in~\eqref{eq:heuristicDeePC} plays the role of either $\eps\cdot\sup_{\theta\in\Theta_2}\|\theta\|_{\infty}$ or $\eps\lambda_{\textup{ini}}$. Thus, the result of Theorem~\ref{thm:reformulation} gives insight into the design of the regularizer on the decision variable $g$ to achieve various robustness goals for DeePC. The results above give rise to the robust DeePC algorithm.
\begin{algorithm}
\caption{Robust DeePC}\label{alg:robustDeePC}
\textbf{Input:} trajectory $\col(\hat{u},\hat{y})\in\real^{(m+p)T}$ with $\hat{u}$ persistently exciting of order $\Tini+\Tf+n$, most recent input/output measurements $\col(\hat{u}_{\textup{ini}},\hat{y}_{\textup{ini}})\in\real^{(m+p)\Tini}$
\begin{enumerate}
\item Set $h^{\star}$ equal to
\begin{align*}
&\underset{h\in H}{\argmin}\left\{c(\hat{\xi},h)\right.\\
&\left.+\eps\cdot \max\left\{\sup_{\theta\in\Theta_2}\|\theta\|_{\infty}\|\col(g,0)\|_{\ast},\lambda_{\textup{ini}}\|h\|_{\ast}\right\}\right\}.
\end{align*}
\label{step:DeePCbegin}
\item Compute the optimal input sequence $u^{\star}=\Ufhat g^{\star}$, where $h^{\star}=\col(g^{\star},-1)$.
\item Apply input $(u(t),\dots,u(t+s))=(u_0^{\star},\dots,u_s^{\star})$ for some $s\leq \Tf-1$.
\item Set $t$ to $t+s$ and update $\hat{u}_{\textup{ini}}$ and $\hat{y}_{\textup{ini}}$ to the $\Tini$ most recent input/output measurements.\label{step:DeePCend}
\item Return to~\ref{step:DeePCbegin}.
\end{enumerate}
\end{algorithm}
\begin{remark}\label{rem:radius}
The Wasserstein ball radius $\eps(\beta)$ in Theorem~\ref{thm:performanceguarantee} is often larger than necessary, i.e., $\Prob\not\in B_{\eps}(\widehat{\Prob})$ with probability much less than $\beta$. Furthermore, even when $\Prob\not\in B_{\eps}(\widehat{\Prob})$, the robust quantity $\sup_{Q\in B_{\eps}(\widehat{\Prob})}\Exp_Q\{c(\xi,h)\}$ may still serve as an upper bound for the out-of-sample performance $\Exp_{\Prob}\{c(\xi,h)\}$~\cite{PME-DK:18}. Thus, for practical purposes, one should choose the radius $\eps$ of the Wasserstein ball in a data-driven fashion (see Section~\ref{sec:simulations}).\oprocend
\end{remark}
\begin{remark}
The assumption that $\Theta_2=\{\theta\mid f_2^{\ast}(\theta)<\infty\}$ is a bounded set in $\real^{p\Tf}$ holds for many objective functions of practical interest. Indeed, any arbitrary norm satisfies this assumption, and in fact any Lipschitz continuous function.\oprocend
\end{remark}
\begin{remark}
The norm $\sum_{j=1}^{p(\Tini+\Tf)}\|\xi_j\|$ used in the Wasserstein metric can be chosen to achieve various robustness goals. Indeed, being robust in the trajectory space in the $\|\cdot\|_{\infty}$ sense requires regularizing with a $\|\cdot\|_1$. Likewise, being robust in the trajectory space in the $\|\cdot\|_{2}$ sense requires regularizing with a $\|\cdot\|_2$.
If no noise is present in the system, we may set $\eps=0$ recovering the DeePC algorithm for deterministic LTI systems.\oprocend
\end{remark}

To prove the theorem above, we require the following lemma which is an extension of~\cite[Lemma A.3]{SSA-DK-PME:17} to functions with vector valued inputs:
\begin{lemma}\label{lem:regularizer}
Let $q,r\in\integerspositive$, and $\Omega_j\subseteq \real^q$ for $j\in\{1,\dots,r\}$. Let $\hat{\zeta}_j\in \Omega_j$ be given for $j\in\{1,\dots,r\}$, $\lambda\in\realpositive$ and $L\colon \real^r\to\real$ convex and continuous such that $\{\theta\mid L^{\ast}(\theta)<\infty\}$ is a bounded set in $\real^{r}$. Then, for fixed $b\in\real^q$,
\begin{align*}
&\underset{\underset{\forall j\in\{1,\dots,r\}}{\zeta_j\in \Omega_j}}\sup L(\zeta_1^Tb,\dots,\zeta_r^Tb)-\lambda\sum_{j=1}^{r}\|\zeta_j-\hat{\zeta}_j\| \\
&\leq\begin{cases}
L(\hat{\zeta}_1^Tb,\dots,\hat{\zeta}_r^Tb) &\textup{if}\;\sup_{\theta\in\Theta}\|\theta\|_{\infty}\|b\|_{\ast}\leq \lambda\\
\infty &\textup{otherwise},
\end{cases}
\end{align*}
where $\|\cdot\|$ is an arbitrary norm on $\real^q$. Furthermore, the above is an equality when $\Omega_j=\real^q$ for all $j\in\{1,\dots,r\}$.
\end{lemma}
\begin{proof}
By definition of the conjugate function,
\begin{align*}
 L^{\ast}(z)&=\underset{\zeta\in \Omega}\sup \langle z,(\zeta_1^Tb,\dots,\zeta_r^Tb)\rangle-L(\zeta_1^Tb,\dots,\zeta_r^Tb)\\
 &=\underset{s,\zeta\in \Omega}\sup\left\{\sum_{j=1}^rz_j\zeta_j^Tb-L(s)\;\middle\vert \;\underset{\forall j\in\{1,\dots,r\}}{s_j=\zeta_j^T b}\right\},
\end{align*}
where $\zeta=\col(\zeta_1,\dots,\zeta_r)$, $z=\col(z_1,\dots,z_r)$, $s=\col(s_1,\dots,s_r)$, and $\Omega=\Omega_1\times\cdots\times\Omega_r$.
The Lagrangian of the above is given by
\[
\mathscr{L}(s,\zeta,\theta)=\sum_{j=1}^r\left(z_j\zeta_j^Tb+\theta_j(s_j-\zeta_j^T b)\right)-L(s),
\]
where $\theta=\col(\theta_1,\dots,\theta_r)$. By strong duality (see, e.g.,~\cite[Proposition 5.3.1]{DPB:09-book}),
\begin{align*}
L^{\ast}(z)&=\inf_{\theta}\sup_{s,\zeta\in\Omega}\sum_{j=1}^r\left(z_j\zeta_j^Tb+\theta_j(s_j-\zeta_j^T b)\right)-L(s)\\
&=\inf_{\theta}\sup_{\zeta\in\Omega}L^{\ast}(\theta)+\sum_{j=1}^r(z_j-\theta_j)\zeta_j^Tb\\
&=\inf_{\theta}L^{\ast}(\theta)+\sum_{j=1}^r\sup_{\zeta_j\in\Omega_j}(z_j-\theta_j)\zeta_j^Tb.
\end{align*}
Hence, by duality
\begin{align*}
L^{\ast}(z)&=\begin{cases}
\inf_{\theta} L^{\ast}(\theta)\\
\textup{s.t.}\;z_jb=\theta_jb, &\forall j\in\{1,\dots,r\}
\end{cases}\\
&=\begin{cases}
\inf_{\theta\in\Theta} L^{\ast}(\theta)\\
\textup{s.t.}\;z_jb=\theta_jb, &\forall j\in\{1,\dots,r\},
\end{cases}
\end{align*}
where $\Theta=\{\theta\mid L^{\ast}(\theta)<\infty\}$ is the effective domain of $L^{\ast}$. Since $L$ is convex and continuous, the  biconjugate $L^{\ast\ast}$ coincides with the function $L$ itself. Hence,
\begin{align*}
L(\zeta_1^Tb,\dots,\zeta_r^Tb)&=\sup_{z} \langle z,(\zeta_1^Tb,\dots,\zeta_r^Tb)\rangle-L^{\ast}(z)\\
&=\begin{cases}
\underset{z}\sup\langle z,(\zeta_1^Tb,\dots,\zeta_r^Tb)\rangle-\underset{\theta\in\Theta}\inf L^{\ast}(\theta)\\
\textup{s.t.}\;z_j b=\theta_j b, \quad\forall j\in\{1,\dots,r\}
\end{cases}\\
&=\sup_{\theta\in\Theta}\langle \theta,(\zeta_1^Tb,\dots,\zeta_r^Tb)\rangle-L^{\ast}(\theta).
\end{align*}
Thus,
\begin{align*}
&\sup_{\zeta\in\Omega}L(\zeta_1^Tb,\dots,\zeta_r^Tb)-\lambda\sum_{j=1}^r\|\zeta_j-\hat{\zeta}_j\|\\
&=\sup_{\zeta\in\Omega}\sup_{\theta\in\Theta}\langle\theta,(\zeta_1^Tb,\dots,\zeta_r^Tb)\rangle-L^{\ast}(\theta)-\lambda\sum_{j=1}^r\|\zeta_j-\hat{\zeta}_j\|\\
&=\sup_{\zeta\in\Omega}\sup_{\theta\in\Theta}\inf_{\underset{\forall j\in\{1,\dots,r\}}{\|\mu_j\|_{\ast}\leq\lambda}}\sum_{j=1}^r\left(\theta_j\zeta_j^Tb-\mu_j^T(\zeta_j-\hat{\zeta}_j)\right)-L^{\ast}(\theta),
\end{align*}
where the last equality comes from the definition of the dual norm and using homogeneity of the norm. Using the minimax theorem (see, e.g.,~\cite[Proposition 5.5.4]{DPB:09-book}) we switch the supremum and infimum in the above and bring it into the sum giving
\begin{align*}
&\sup_{\zeta\in\Omega}L(\zeta_1^Tb,\dots,\zeta_r^Tb)-\lambda\sum_{j=1}^r\|\zeta_j-\hat{\zeta}_j\|\\
&=\sup_{\theta\in\Theta}\inf_{\underset{\forall j\in\{1,\dots,r\}}{\|\mu_j\|_{\ast}\leq\lambda}}-L^{\ast}(\theta)+\sum_{j=1}^r\sup_{\zeta_j\in\Omega_j}\theta_j\zeta_j^Tb-\mu_j^T(\zeta_j-\hat{\zeta}_j)\\
&\leq \sup_{\theta\in\Theta}\inf_{\underset{\forall j\in\{1,\dots,r\}}{\|\mu_j\|_{\ast}\leq\lambda}}-L^{\ast}(\theta)+\sum_{j=1}^r\sup_{\zeta_j\in\real^q}\theta_j\zeta_j^Tb-\mu_j^T(\zeta_j-\hat{\zeta}_j),
\end{align*}
where equality holds when $\Omega_j=\real^q$ for all $j\in\{1,\dots,r\}$. Hence,
\begin{align*}
&\sup_{\zeta\in\Omega}L(\zeta_1^Tb,\dots,\zeta_r^Tb)-\lambda\sum_{j=1}^r\|\zeta_j-\hat{\zeta}_j\|\\
&\leq\sup_{\theta\in\Theta}\inf_{\underset{\forall j\in\{1,\dots,r\}}{\|\mu_j\|_{\ast}\leq\lambda}}\begin{cases}
\sum_{j=1}^r(\mu_j^T\hat{\zeta}_j)-L^{\ast}(\theta) &\textup{if}\;\underset{\forall j\in\{1,\dots,r\}}{\mu_j=\theta_jb}\\
\infty &\textup{otherwise}
\end{cases}\\
&=\sup_{\theta\in\Theta}\begin{cases}
\sum_{j=1}^r\langle \theta_jb,\hat{\zeta}_j\rangle-L^{\ast}(\theta) &\textup{if}\;\underset{\forall j\in\{1,\dots,r\}}{\|\theta_jb\|_{\ast}\leq\lambda}\\
\infty &\textup{otherwise}
\end{cases}\\
&=\begin{cases}
L(\hat{\zeta}_1^Tb,\dots,\hat{\zeta}_r^Tb) &\textup{if}\;\sup_{\theta\in\Theta}\|\theta\|_{\infty}\|b\|_{\ast}\leq\lambda\\
\infty &\textup{otherwise}.
\end{cases}
\end{align*}
This proves the claimed result.
\end{proof}
\begin{remark}
Lemma~\ref{lem:regularizer} may also be used to extend the results of~\cite{SSA-DK-PME:17} to learning problems in which the input data is matrix-valued. We do not explore this connection as this is beyond the scope of this paper.\oprocend
\end{remark}
\begin{proof}\longthmtitle{Theorem~\ref{thm:reformulation}}
Since $c(\cdot,h)$ is a proper, convex, continuous function for all $h\in H$ then
\begin{align*}
&\underset{Q\in B_{\eps}(\widehat{\Prob})}{\text{sup}}\Exp_{Q}\left\{c(\xi,h)\right\}\\
&=\underset{\lambda\geq 0}\inf\;\lambda\eps + \underset{\xi\in\Xi}\sup(c(\xi,h)-\lambda\|\xi-\hat{\xi}\|_{\textup{W}}).
\end{align*}
This can be shown by studying the dual problem, and noticing that the worst case distribution coincides with a Dirac distribution at the point $\xi$ which results in the largest norm $\|\xi-\hat{\xi}\|_{\textup{W}}$ (see~\cite[Theorem 6.3]{PME-DK:18}).
Define $f_3(\cdot)\defeq\lambda_{\textup{ini}}\|\cdot\|_1$. By separability of the objective function,
\begin{align*}
&\underset{Q\in B_{\eps}(\widehat{\Prob})}{\text{sup}}\Exp_{Q}\left\{c(\xi,h)\right\}\\
&=\underset{\lambda\geq 0}\inf\;\lambda\eps + f_1((\Uf,0) h)\\
&+\sup_{\xi}\{f_2(\xi_{p\Tini+1}h,\dots,\xi_{p(\Tini+\Tf)}h) \\
&+f_3(\xi_1h,\dots,\xi_{p\Tini}h)-\lambda\|\xi-\hat{\xi}\|_{\textup{W}})\}\\
&=\begin{cases}
\underset{\lambda\geq 0,s_1,s_2}\inf\;\lambda\eps + f_1((\Uf,0) h)+ s_2+s_3\\
\textup{s.t.}\;\underset{\xi\in\Xi}\sup f_2(\xi_{p\Tini+1}h,\dots,\xi_{p(\Tini+\Tf)}h)\\
\quad\quad\quad -\sum_{j=p\Tini+1}^{p(\Tini+\Tf)}\lambda\|\xi_j-\hat{\xi}_j\|\leq s_2\\
\quad\;\underset{\xi\in\Xi}\sup f_3(\xi_1h,\dots,\xi_{p\Tini}h)\\
\quad\quad\quad -\sum_{j=1}^{p\Tini}\lambda\|\xi_j-\hat{\xi}_j\|\leq s_3,
\end{cases}
\end{align*}
where we used the epigraph formulation and the definition of $\|\cdot\|_{\textup{W}}$ for the last equality.
By definition of the conjugate function, $f_3^{\ast}(\theta)=0$ if $\|\theta\|_{\infty}\leq \lambda_{\textup{ini}}$ and is infinite otherwise. Hence, $\Theta_3\defeq\{\theta\mid f_3^{\ast}(\theta)<\infty\}$ is a bounded set in $\real^{p\Tf}$. We also have that $\Theta_2\defeq\{\theta\mid f_2^{\ast}(\theta)<\infty\}$ by assumption. Hence, by Lemma~\ref{lem:regularizer},
\begin{align*}
&\underset{\xi\in\Xi}\sup f_3(\xi_1h,\dots,\xi_{p\Tini}h)-\sum_{j=1}^{p\Tini}\lambda\|\xi_j-\hat{\xi}_j\|\\
&\leq
\begin{cases}
f_3(\hat{\xi}_1h,\dots,\hat{\xi}_{p\Tini}h)  &\textup{if}\; \underset{\theta\in\Theta_3}\sup\|\theta\|_{\infty} \|h\|_{\ast}\leq\lambda\\
\infty &\textup{otherwise},
\end{cases}
\end{align*}
with equality when for all $j\in\{1,\dots,p\Tini\}$, $\Xi_j=\real^{T-\Tini-\Tf+2}$. Note that the uncertainties $\xi_{p\Tini+1},\dots,\xi_{p(\Tini+\Tf)}$ affecting $f_2$ have 0 as their last entry by definition. Hence, carrying through the steps of the proof of Lemma~\ref{lem:regularizer} for $f_2$ yields
\begin{align*}
&\underset{\xi\in\Xi}\sup f_2(\xi_{p\Tini+1}h,\dots,\xi_{p(\Tini+\Tf)}h)-\sum_{j=p\Tini+1}^{p(\Tini+\Tf)}\lambda\|\xi_j-\hat{\xi}_j\|\\
&\leq
\begin{cases}
f_2(\hat{\xi}_{p\Tini+1}h,\dots,\hat{\xi}_{p(\Tini+\Tf)}h)  &\textup{if}\; \underset{\theta\in\Theta_2}\sup\|\theta\|_{\infty} \|\tilde{h}\|_{\ast}\leq\lambda\\
\infty &\textup{otherwise},
\end{cases}
\end{align*}
where $\tilde{h}=\col(g,0)$ and equality holds when $\Xi_j=\real^{T-\Tini-\Tf+1}\times\{0\}$ for all $j\in\{p\Tini+1,\dots,p(\Tini+\Tf)\}$. By reversing the epigraph formulation, we have
\begin{align*}
&\underset{Q\in B_{\eps}(\widehat{\Prob})}{\text{sup}}\Exp_{Q}\left\{c(\xi,h)\right\}\\
&\leq\begin{cases}
\underset{\lambda\geq 0}\inf\;\lambda\eps + f_1((\Uf,0) h)+ f_2(\hat{\xi}_{p\Tini+1}h,\dots,\hat{\xi}_{p(\Tini+\Tf)}h)\\
\quad+ f_3(\hat{\xi}_1h,\dots,\hat{\xi}_{p\Tini}h)\\
\textup{s.t.}\;\sup_{\theta\in\Theta_2}\|\theta\|_{\infty} \|\tilde{h}\|_{\ast}\leq\lambda\\
\quad\;\;\sup_{\theta\in\Theta_3}\|\theta\|_{\infty} \|h\|_{\ast}\leq\lambda.
\end{cases}
\end{align*}
Note that the infinite case is dropped since we are taking the infimum over $\lambda\geq 0$. Substituting notation yields
\begin{align*}
&\underset{Q\in B_{\eps}(\widehat{\Prob})}{\text{sup}}\Exp_{Q}\left\{c(\xi,h)\right\}\\
&\leq c(\hat{\xi},h)+\eps\cdot \max\left\{\sup_{\theta\in\Theta_2}\|\theta\|_{\infty}\|\tilde{h}\|_{\ast},\lambda_{\textup{ini}}\|h\|_{\ast}\right\}.
\end{align*}
Noting that $\tilde{h}=\col(g,0)$ gives the claimed result.
\end{proof}
%
\section{Simulations}\label{sec:simulations}
We illustrate the performance of the Robust DeePC Algorithm~\ref{alg:robustDeePC} on a model of a quadcopter linearized around the hover position. The states of the quadcopter model are given by the 3 spatial coordinates ($x$, $y$, $z$) and their velocities, and  the 3 angular coordinates $(\alpha,\beta,\gamma)$ and their velocities, i.e., the state is $(x,y,z,\dot{x},\dot{y},\dot{z},\alpha,\beta,\gamma,\dot{\alpha},\dot{\beta},\dot{\gamma})$. The inputs are given by the individual thrusts from the 4 rotors. Full state measurement was assumed. The states are affected by additive zero-mean Gaussian noise. The state-space matrices used in~\eqref{eq:stochasticsystem} are
\setcounter{MaxMatrixCols}{20}
\[
\scalemath{0.7}{A=\begin{pmatrix}
1 & 0 & 0 & 0.1 & 0 & 0 & 0 & 0.049 & 0 & 0 & 0.0016 & 0 \\
0 &1 &0 &0 &0.1 &0 &-0.049 &0 &0 &-0.0016 &0 &0\\
0 &0 &1 &0 &0 &0.1 &0 &0 &0 &0 &0 &0\\
0 &0 &0 &1 &0 &0 &0 &0.981 &0 &0 &0.049 &0\\
0 &0 &0 &0 &1 &0 &-0.981 &0 &0 &-0.049 &0 &0\\
0 &0 &0 &0 &0 &1 &0 &0 &0 &0 &0 &0\\
0 &0 &0 &0 &0 &0 &1 &0 &0 &0.1 &0 &0\\
0 &0 &0 &0 &0 &0 &0 &1 &0 &0 &0.1 &0\\
0 &0 &0 &0 &0 &0 &0 &0 &1 &0 &0 &0.1\\
0 &0 &0 &0 &0 &0 &0 &0 &0 &1 &0 &0\\
0 &0 &0 &0 &0 &0 &0 &0 &0 &0 &1 &0\\
0 &0 &0 &0 &0 &0 &0 &0 &0 &0 &0 &1
\end{pmatrix},
}
\]
\[
\scalemath{0.7}{B=\begin{pmatrix}
-2.3\times 10^{-5} & 0 & 2.3\times 10^{-5} & 0 \\
0 & -2.3\times 10^{-5} & 0 & 2.3\times 10^{-5} \\
1.75\times 10^{-2} & 1.75\times 10^{-2} & 1.75\times 10^{-2} & 1.75\times 10^{-2} \\
-9.21\times 10^{-4} & 0 &9.21\times 10^{-4} & 0 \\
0 & -9.21\times 10^{-4} & 0 & 9.21\times 10^{-4} \\
0.35 & 0.35 & 0.35 & 0.35 \\
0 & 2.8\times 10^{-3} & 0	 & -2.8\times 10^{-3} \\
-2.8\times 10^{-3} & 0 & 2.8\times 10^{-3} & 0 \\
3.7\times 10^{-3} & -3.7\times 10^{-3} & 3.7\times 10^{-3} & -3.7\times 10^{-3} \\
0 & 5.6\times 10^{-2} & 0 & -5.6\times 10^{-2} \\
-5.6\times 10^{-2} & 0 & 5.6\times 10^{-2} & 0 \\
7.3\times 10^{-2} & -7.3\times 10^{-2} & 7.3\times 10^{-2} & -7.3\times 10^{-2}
\end{pmatrix},
}
\]
\[
\scalemath{0.7}{C=I_{12\times 12},\; D=0_{12\times 4},\; E=(I_{12\times 12},0_{12\times 12}),\; F=(0_{12\times 12},I_{12\times 12}).
}
\]
During the data collection process, the input was drawn from a uniform random variable to ensure persistency of excitation. The resulting output data was corrupted by additive noise drawn from a zero-mean Gaussian distribution.
We used 214 input/output measurements to populate the matrices $\Uphat$, $\Yphat$, $\Ufhat$, and $\Yfhat$, which is the minimum number of measurements needed to ensure persistency of excitation. We commanded the quadcopter to track a parameterized figure-8 reference trajectory, denoted by $r$.
 \begin{figure}[t!]
	\centering
		\includegraphics[width=0.9\linewidth]{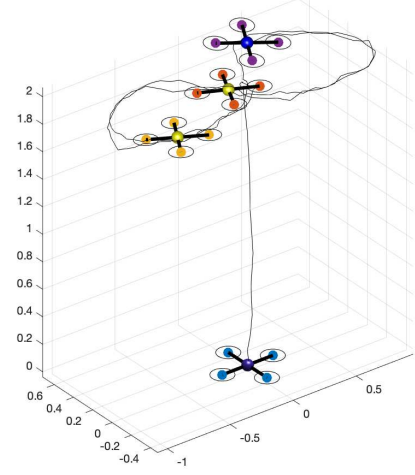}
	\caption{Quadcopter controlled using the Robust DeePC Algorithm~\ref{alg:robustDeePC} following a figure-8 trajectory. The quadcopter trajectory is shown in black.}
	\label{fig:itworksfigure8}
\end{figure}
\begin{figure}[t!]
	\centering
		\includegraphics[width=0.9\linewidth]{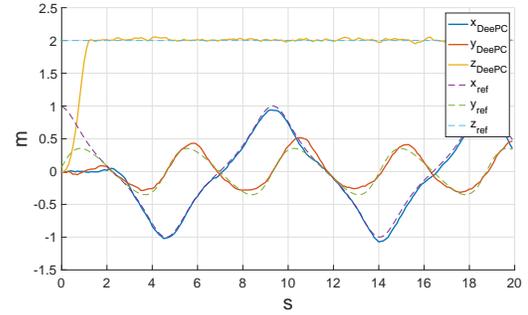}
	\caption{Position of quadcopter following a figure-8 trajectory.}
	\label{fig:itworksxyz}
\end{figure}
\begin{figure}[t!]
	\centering
		\includegraphics[width=0.9\linewidth]{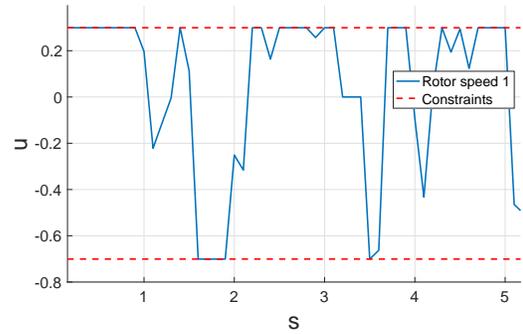}
	\caption{First element of the input signal of the quadcopter following a figure-8 trajectory. The red dashed lines represent constraints.}
	\label{fig:itworksu}
\end{figure}

The following parameters were chosen for the optimization problem~\eqref{eq:optwasserstein}: the infinity-norm $\|\cdot\|_{\infty}$ was used in the definition of the the norm $\|\cdot\|_{\textup{W}}$ used in the Wasserstein metric, $\Tf=30$, $\Tini =1$, $f(u,y) = \|u\|_1+200\|y-r\|_1$, $\lambda_{\textup{ini}}=10^5$, $\eps=0.001$, $\mathcal{U}=[-0.7007,0.2993]^{\Tf}$.
The input constraint set $\mathcal{U}$ is chosen this way to mimic the constraints present in a nonlinear quadcopter model where the normalized rotor thrusts can only vary in the set $[0,1]$. Online measurement noise was drawn from same Gaussian distribution. The performance of the algorithm is seen in Figures~\ref{fig:itworksfigure8}-\ref{fig:itworksu}. As can be seen, the algorithm exhibits desirable behaviour despite the measurement noise and process noise.
 
As discussed in Remark~\ref{rem:radius}, computing an optimal radius $\eps$ for the Wasserstein ball in a data-driven fashion may increase the performance of the algorithm. We present results of the robust algorithm for various Wasserstein radii when the quadcopter was commanded to follow a step trajectory. By observing the performance of the quadcopter, we are able to estimate a range of optimal Wasserstein radii. We again measured the full state. The states were affected by zero-mean truncated Gaussian noise, where the truncation was made 3 standard deviations from both sides of the mean. The output measurements were also corrupted by a zero-mean truncated Gaussian distribution. The infinity-norm was used in the definition of the Wasserstein metric. A Wasserstein radius $\eps$ was fixed and the cost $f(u,y)$ accumulated over the 20 second horizon of the simulation was computed. The cost was averaged over 15 random simulations for each $\eps$. 

The results shown in Figure~\ref{fig:costvseps} indicate that $\eps$ should be chosen in the interval $[0.001,0.01]$. The results also support the claim that optimizing the in-sample performance~\eqref{eq:insample} (i.e., setting $\eps=0$) displays poor out-of-sample performance, hence justifying the robust approach. It is also clear from Figure~\ref{fig:costvseps} that choosing $\eps$ large (i.e., being over-conservative) results in poor performance.
\begin{figure}[htb!]
	\centering
		\includegraphics[width=\linewidth]{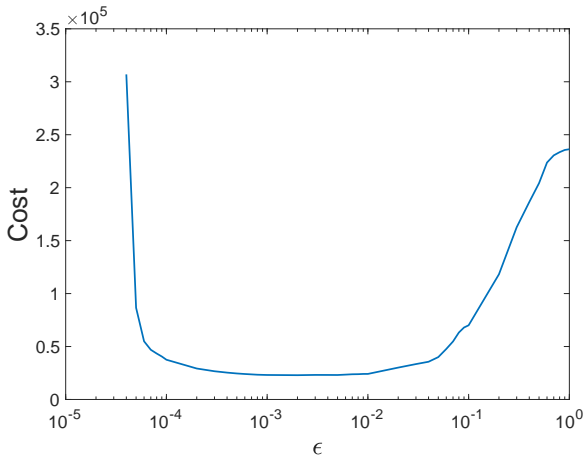}
	\caption{Performance of the Robust DeePC Algorithm~\ref{alg:robustDeePC} for many Wasserstein radii.}
	\label{fig:costvseps}
\end{figure}
%

\section{Conclusion}\label{sec:conclusion}
We studied the problem of controlling an unknown stochastic system with the goal of minimizing an objective function dependent on the input/output trajectories of the system. With no knowledge of the nature of the stochasticity, we proposed a robust DeePC algorithm which uses noise corrupted input/output data to compute optimal and robust control inputs. Robustifying the original DeePC algorithm gave rise to principled regularization terms, supporting the observed superior performance of the regularized algorithm on stochastic systems. Future work includes incorporating multiple measured data sets to improve performance, and including output constraints.


\bibliographystyle{IEEEtran}%
\bibliography{IEEEabrv,JC}

\end{document}